\documentclass[11pt, reqno]{amsart}
\usepackage{graphicx} 
\usepackage[english]{babel}
\usepackage[T1]{fontenc}
\usepackage[utf8]{inputenc}

\usepackage{mathrsfs}
\usepackage{amsmath}
\usepackage{amssymb}
\usepackage{amsfonts}
\usepackage{amsthm}
\usepackage{amscd}%
\usepackage{graphicx}
\usepackage{setspace}
\usepackage[T1]{fontenc}
\usepackage{geometry}
\usepackage{mathtools}
\usepackage{faktor}
\usepackage[dvipsnames]{xcolor}
\usepackage[colorlinks,citecolor=NavyBlue,urlcolor=NavyBlue]{hyperref}
\usepackage{cleveref}
\usepackage[all]{xy}
\usepackage{enumitem}
\usepackage{comment}
\usepackage{ dsfont }

\makeatletter
\def\saveenum{\xdef\@savedenum{\the\c@enumi\relax}}
\def\resetenum{\global\c@enumi\@savedenum}
\makeatother

\numberwithin{equation}{section}

\newcommand{\bbN}{\mathbb{N}}

\newcommand{\cC}{\mathcal{C}}

\newcommand{\Es}{E^s}

\newcommand{\Eu}{E^u}

\newcommand{\Ws}{W^s}

\newcommand{\Wu}{W^u}

\newcommand{\inv}{^{-1}}

\newcommand{\intert}{\cap\kern-0.7em|\kern0.7em}

\newcommand{\bbP}{\mathbb{P}}

\newcommand{\norm}[1]{
	\lVert #1 \rVert
}

\newtheorem{thm}{Theorem}[section]
\newtheorem{cor}[thm]{Corollary}
\newtheorem{lemma}[thm]{Lemma}
\newtheorem{prop}[thm]{Proposition}

\newtheorem{conjecture}[thm]{\textbf{Conjecture}}

\newtheorem{defn}[thm]{Definition}

\theoremstyle{remark}

\newcommand{\simh}{\overset{h}{\sim}}

\newcommand{\CP}{\text{CP}(\sigma, \tilde{\sigma}_1, \tilde{\sigma}_2,  \rho, \tilde{\rho}_1, \tilde{\rho}_2, \eta)}

\newcounter{mysubequations}

\title{Finitness of measured homoclinic classes with large Lyapunov exponents for $\cC^2$ surface diffeomorphisms}
\author{Ghezal Matéo}
\date{}

\begin{document}
	
	\maketitle
	
	\bigskip
	
	\begin{center}
		\textbf{Abstract}
	\end{center}

	\bigskip
	
	In this paper, we show the finiteness of homoclinic classes carrying measures with large Lyapunov exponents for $\cC^2$ surface diffeomorphisms. As a consequence, we derive the finiteness of the set of ergodic measures of maximal entropy, in the case where the entropy of the system is large.
	
	\section{Introduction}\label{sec:intro}
	
	A famous theorem of Newhouse states that for a $\cC^{\infty}$ diffeomorphism on a compact, connected, boundaryless manifold, there exists an ergodic measure of maximal entropy \cite{newhouse1989continuity}. He asks whether the set of these measures is finite for surface diffeomorphisms with positive entropy.
	
	In this paper, we provide a quantitative answer to Newhouse's question for $\cC^2$ diffeomorphisms with large entropy. In \cite{buzzi2022measures}, Buzzi, Crovisier, and Sarig answered this question affirmatively using Yomdin theory. We use a simpler approach, relying on a theorem of Crovisier and Pujals \cite{crovisier2018strongly} and Pliss's lemma, to show that ergodic measures with entropy bounded away from 0 give positive mass to a non-uniformly hyperbolic set. Our approach establishes the finiteness of measured homoclinic classes with sufficiently large Lyapunov exponents, which is not present in \cite{buzzi2022measures}.
	
	\subsection{Main Result}
	
	In this paper, $M$ denotes a closed surface: a compact, connected, boundaryless, two-dimensional Riemannian manifold, and $d(.,.)$ denotes the associated distance. Let $f : M \rightarrow M$ be a $\cC^2$ diffeomorphism. Define the following quantity:
	\begin{equation*}
		R(f) = \max\big(\lim_{n \rightarrow +\infty}\frac{1}{n}\log\norm{df^n},\lim_{n \rightarrow +\infty}\frac{1}{n}\log\norm{df^{-n}}\big).
	\end{equation*}
	
	Let $\mu_1$ and $\mu_2$ be two hyperbolic invariant probability measures. We say that $\mu_1$ and $\mu_2$ are homoclinically related if their stable and unstable manifolds intersect transversely on a set of positive measure for each measure. See Section \ref{sec:pesin} for a more precise definition.\\
	
	We can now state our main theorem:
	\begin{thm}
		Let $f$ be a $\cC^2$ diffeomorphism of a closed surface. Let $(\mu_n)_{n \in \bbN}$ be any sequence of ergodic hyperbolic saddle measures such that $\mu_i$ is not homoclinically related to $\mu_j$ for $i\ne j$. Denote by $\lambda^u(\mu_i)>0>\lambda^s(\mu_i)$ their Lyapunov exponents. Then the number of measures $\mu_i$ satisfying:
		\begin{equation*}
			\min\big(\lambda^u(\mu_i),-\lambda^s(\mu_i)\big)>\frac{19}{20}R(f)
		\end{equation*}
		is finite.
		\label{thm:MainThm}
	\end{thm}
	
	We then obtain the following corollary regarding the number of measures of maximal entropy.
	\begin{cor}
		Let $f$ be a $\cC^2$ diffeomorphism of a closed surface. If:
		\begin{equation}
			h_{top}(f) > \frac{19}{20} R(f)
			\label{hypmainthm}
		\end{equation}
		then the number of ergodic measures of maximal entropy is finite.
		\label{cor:FinMME}
	\end{cor}
	
	The constant $\frac{19}{20}$ is not optimal: Buzzi, Crovisier, and Sarig \cite{buzzi2022measures} proved that this constant can be replaced by $\frac{1}{2}$, which is obviously much better. In fact, \cite{buzzi2022measures} conjectured that this constant is sharp. However, the approach we use allows us to prove a different result, interesting because our assumption is weaker in the sense that we only need a lower bound on Lyapunov exponents.\\
	
	This result also extends to non-compact surfaces, possibly with boundaries, having a global attractor.
	
	\subsection{A Conjecture of Crovisier on Measures with Large Exponents}
	
	Let us attempt to motivate what follows and explain why this result is not merely a weaker version of \cite{buzzi2022measures}. In this paper, we prove Theorem \ref{thm:FinitenessHomClas}, which states that for any $\cC^2$ diffeomorphism $f$ of a closed surface $M$, the set of measured homoclinic classes (see Section \ref{sec:pesin} for definitions) carrying measures with Lyapunov exponents bounded away from 0 by $\frac{19}{20}R(f)$ is finite. In \cite{buzzi2022measures}, a similar statement is shown, replacing our condition on Lyapunov exponents with a condition on entropy.
	\begin{thm}[Buzzi-Crovisier-Sarig, \cite{buzzi2022measures}]
		Let $f$ be a $\cC^r$ diffeomorphism of a closed surface $M$. Then the number of homoclinic classes with entropy greater than $R(f)/r$ is finite.
	\end{thm}
	
	In particular, our condition allows us to address measures with small entropy but sufficiently large Lyapunov exponents. Ruelle's inequality (see Section \ref{sec:pesin} or \cite{ruelle1978inequality}) shows that the condition on Lyapunov exponents implies the one on entropy. Sylvain Crovisier then conjectured the following:
	\begin{conjecture}
		Let $f$ be a $\cC^r$ diffeomorphism and let $(\mu_n)_{n\geq1}$ be any sequence of saddle ergodic invariant measures such that $\mu_i$ is not homoclinically related to $\mu_j$ for $i\ne j$. Denote by $\lambda^u(\mu_i)>0>\lambda^s(\mu_i)$ their Lyapunov exponents. Then the number of measures $\mu_n$ satisfying:
		\begin{equation*}
			\min\big(\lambda^u(\mu_n),-\lambda^s(\mu_n)\big) >\frac{R(f)}{r}
		\end{equation*}
		is finite.
	\end{conjecture}
	
	Our result is a step toward this conjecture, replacing $1/r$ with a larger constant.
	
	\subsection{Other Approaches and Related Work}
	
	One of the main purposes of this paper is to prove a finiteness result for certain classes of measures with large entropy (or exponents). There are different approaches to achieving such results. The first one, used in \cite{buzzi2022measures}, is a topological approach. Another, more quantitative approach, involves studying the continuity of Lyapunov exponents, as explored in \cite{buzzi2022continuity}. The approach used in this paper exploits the idea that sufficiently large Lyapunov exponents allow control over the geometry of stable and unstable manifolds on a uniform set of positive measure.\\
	
	In a recent work, C. Luo and D. Yang \cite{luo2024ergodic} proved that the set of points with a uniformly large local unstable manifold has uniform measure for all invariant ergodic measures with sufficiently large entropy. They used the Yomdin-Burguet reparametrization lemma, which is more intricate than our approach. It seems that their result also implies the finiteness of measures of maximal entropy for diffeomorphisms with sufficiently large entropy, yielding a better bound than ours.
	
	\subsection{Outline of the Paper}
	In Section \ref{sec:pesin}, we recall key aspects of Pesin theory and introduce homoclinic relations for measures.
	
	The third section is devoted to proving the main proposition, Proposition \ref{prop:MainProp}. After stating  Crovisier-Pujals stable manifold theorem, we provide the proof of Proposition \ref{prop:MainProp} using Pliss lemma.
	
	In Section \ref{sec:mme}, we derive Theorem \ref{thm:FinitenessHomClas}, which establishes the finiteness of homoclinic classes containing measures with large exponents, as a consequence of Proposition \ref{prop:MainProp}.
	
	\subsection{Acknowledgments}
	We thank Sylvain Crovisier, who proposed the problem to the author, for very helpful discussion and for his comments on previous versions of this note.
	
	\section{Homoclinic Relations and Pesin Theory}\label{sec:pesin}
	
	In this section, $f$ is a $\cC^r, \ r>1$, diffeomorphism of a compact, connected, boundaryless manifold $M$ of dimension $2$. We refer to \cite{katok1995introduction}, Chapter S, for background on Pesin theory.
	
	\medskip
	
	\subsection{Hyperbolic Measures}
	
	Denote by $\bbP(f)$ the set of $f$-invariant probability measures. We begin by recalling Oseledets's theorem. Given $\mu \in \bbP(f)$, there exists a measurable $df$-invariant splitting $T_xM = \Es_x \oplus \Eu_x$ for $\mu$-almost every $x \in M$. This splitting satisfies the following property: there exist real numbers $\lambda^s(x,\mu) \leq \lambda^u(x,\mu)$ such that for all $v \in E^{u/s}_x$, we have:
	\begin{equation*}
		\lim_{n\rightarrow\pm \infty} \frac{1}{|n|}\log \norm{d_xf^n v} = \lambda^{u/s}(x,\mu).
	\end{equation*}
	These numbers are called the Lyapunov exponents of $\mu$ and are constant almost everywhere when $\mu$ is ergodic. In this case, we do not explicitly write the dependence on $x$.\\
	
	An invariant measure $\mu$ is said to be hyperbolic if its Lyapunov exponents are nonzero $\mu$-almost everywhere. A measure is of saddle type if $\lambda^s(x,\mu) < 0 < \lambda^u(x,\mu)$ $\mu$-almost everywhere. We denote by $\bbP^h(f)$ the set of hyperbolic $f$-invariant measures and by $\bbP^s(f)$ the set of saddle-type $f$-invariant measures. If we restrict to ergodic measures, we write $\bbP^{h/s}_e(f)$. Ruelle's inequality bounds the Lyapunov exponents in terms of entropy:
	
	\begin{thm}[Ruelle's Inequality]
		Let $\mu \in \bbP_e^s(f)$. Then:
		\begin{equation*}
			h(f,\mu) \leq \min\big(-\lambda^s(\mu),\lambda^u(\mu)\big).
		\end{equation*}
	\end{thm}
	
	\subsection{Pesin Blocks and Invariant Manifolds}
	We recall some fundamental results from Pesin theory. Given a hyperbolic saddle ergodic measure $\mu \in \bbP_e^s(f)$, there exists a family of compact measurable sets $(K_n)_{n \in \bbN}$, called Pesin blocks, with the following properties:
	
	\begin{itemize}[label={--}]
		\item For each $n$, $f\inv(K_n) \cup K_n \cup f(K_n) \subset K_{n+1}$, ensuring that the measurable set $Y = \cup_n K_n$ is invariant.
		\item Each $K_n$ admits a continuous splitting $TM_{|K_n} = \mathcal{E}^s \oplus \mathcal{E}^u$, and for any saddle hyperbolic measure $\mu$, we have $E^{u/s}_x = \mathcal{E}^{u/s}_x$ for $\mu$-almost every $x \in K_n$.
		\item There exist two continuous ($\cC^r$) families of embedded discs $(\Ws_{loc}(x))_{x \in K_n}$ and $(\Wu_{loc}(x))_{x \in K_n}$, called stable and unstable local manifolds, tangent to $\mathcal{E}^s$ and $\mathcal{E}^u$. These are $f$-invariant: $f(\Ws_{loc}(x)) \subset \Ws_{loc}(f(x))$ and $f\inv(\Wu_{loc}(x)) \subset \Wu_{loc}(f\inv(x))$.
	\end{itemize}
	
	\subsection{Homoclinic Relations}
	
	Following Newhouse \cite{newhouse1972hyperbolic} and Buzzi-Crovisier-Sarig \cite{buzzi2022measures}, we define homoclinic relations for hyperbolic saddle measures. Let $U,V\subset M$ be two sub-manifolds, we denote by $U\intert V$ the transverse intersection between $U$ and $V$.
	
	\begin{defn}
		Let $\mu_1, \mu_2 \in \mathbb{P}_e^s(f)$. We write $\mu_1 \preceq \mu_2$ if there exist measurable sets $A_1, A_2 \subset M$ such that $\mu_i(A_i) > 0$ and for all $(x_1, x_2) \in A_1 \times A_2$, $\Wu(x_1) \intert \Ws(x_2) \neq \emptyset$.
		
		We say that $\mu_1$ and $\mu_2$ are homoclinically related if $\mu_1 \preceq \mu_2$ and $\mu_2 \preceq \mu_1$, and we write $\mu_1 \simh \mu_2$.
	\end{defn}
	
	\begin{prop}
		The homoclinic relation for ergodic hyperbolic measures is an equivalence relation.
		\label{equivrelation}
	\end{prop}
	
	This is proved using the inclination lemma. A complete proof can be found in \cite[Proposition~2.11]{buzzi2022measures}. For any $\mu \in \bbP_e^s(f)$, we denote its measured homoclinic class by $\mathcal{M}(\mu)$.
	
\section{Main Proposition}\label{sec:main}
Throughout this section, let $f$ be a $\cC^2$ diffeomorphism of a closed surface $M$. The goal of this section is to prove the main proposition of this paper, namely Proposition \ref{prop:MainProp}.

\subsection{Crovisier-Pujals Stable Manifold Theorem}
We first state the Crovisier-Pujals stable manifold theorem, which was proved in \cite{crovisier2018strongly}.

\begin{thm}[Crovisier-Pujals Stable Manifold]
	Let $f$ be a $\cC^2$ diffeomorphism of $M$. Let $\sigma, \tilde{\sigma}, \rho, \tilde{\rho} \in (0,1)$ be such that $\frac{\tilde{\sigma}\tilde{\rho}}{\sigma\rho}>\sigma$. Suppose there exists a point $x \in M$ and a direction $E\subset T_xM$ such that:
	\begin{equation*}
		\forall n\geq0, \quad \tilde{\sigma}^n\leq \norm{d_xf^n_{|E}}\leq \sigma^n \quad \text{and} \quad \tilde{\rho}^n \leq \frac{\norm{d_xf^n_{|E}}^2}{|\det d_xf^n|} \leq \rho^n.
	\end{equation*}
	Then $x$ has a one-dimensional stable manifold that varies continuously with $x$ in the $\cC^1$ topology. 
	\label{thm:CP-StableManifold}
\end{thm}

\begin{defn}[CP-hyperbolic sets]
	Let $\sigma, \tilde{\sigma}_1, \tilde{\sigma}_2, \rho, \tilde{\rho}_1, \tilde{\rho}_2, \eta \in (0,1)$ be such that $\frac{\tilde{\sigma}_i\tilde{\rho}_i}{\sigma\rho}>\sigma$. We define the CP-hyperbolic set, denoted by $\CP$, as the set of all $x \in M$ for which there exist two distinct directions $E,F\subset T_xM$ satisfying, for all $n\geq0$:
	\begin{enumerate}[itemsep=0.2cm]
		\item [(CP1)] $\tilde{\sigma}_1^n\leq \norm{d_xf^n_{|E}}\leq \sigma^n$,
		\item [(CP2)] $\tilde{\rho}_1^n \leq \frac{\norm{d_xf^n_{|E}}^2}{|\det d_xf^n|} \leq \rho^n$,
		\item [(CP3)] $\tilde{\sigma}_2^n\leq \norm{d_xf^{-n}_{|F}}\leq \sigma^n$,
		\item [(CP4)] $\tilde{\rho}_2^n \leq \frac{\norm{d_xf^{-n}_{|F}}^2}{|\det d_xf^{-n}|} \leq \rho^n$,
		\item [(CP5)] $|\cos \measuredangle (E,F)|\leq \eta$.
	\end{enumerate}
	\label{def:CP-NUH-Set}
\end{defn}

Condition $(CP5)$ plays a crucial role. Without it, we could construct a sequence of points $x_n$ in some set $\CP$, with two directions $E_n$ and $F_n$, such that their angle tends to zero. The limit point $x$ (after taking a subsequence) would then have a direction that neither expands nor contracts, contradicting the definition of $\CP$. Condition $(CP5)$ ensures that this situation does not occur and that any CP-hyperbolic set is compact. This is similar to the assumptions on Pesin blocks (see \cite{katok1995introduction}).

\subsection{Main Proposition}
We are now ready to state our main proposition.

\begin{prop}
	There exist an integer $N$ and constants $\sigma, \tilde{\sigma}_1, \tilde{\sigma}_2, \rho, \tilde{\rho}_1, \tilde{\rho}_2, \eta \in (0,1)$ such that $\frac{\tilde{\sigma}_i\tilde{\rho}_i}{\sigma\rho}>\sigma$ and for any ergodic measure $\mu \in \bbP(f)$ satisfying:
	\begin{equation}
		\min\big(-\lambda^s(\mu),\lambda^u(\mu)\big) > \frac{19}{20} R(f),
		\label{eq:mainProp}
	\end{equation}
	we have $\mu(\CP) > 0$, where $\CP$ is a CP-hyperbolic set for $f^N$.
	\label{prop:MainProp}
\end{prop}

It is worth emphasizing that measures with large Lyapunov exponents assign positive measure to the same set, which can be seen as a form of uniformity. This can be compared to the SPR property (see \cite{buzzi2025strong}). The remainder of this section is devoted to the proof of Proposition \ref{prop:MainProp}.
\subsection{Proof of Proposition \ref{prop:MainProp}}
We now state Pliss' lemma, a powerful combinatorial tool for obtaining lower bounds on the measure of certain sets with hyperbolic properties. Since it is well known in the literature, we do not provide a proof, but it can be found in \cite{crovisier2018strongly}.

\begin{lemma}[Pliss' Lemma]
	For any $\alpha_1 < \alpha_2 < \alpha_3$ and any sequence $(a_n)_{n \in \bbN} \in (\alpha_1, +\infty)$ satisfying:
	\begin{equation*}
		\limsup_{n \rightarrow \infty} \frac{1}{n} \sum_{i=0}^{n-1} a_i \leq \alpha_2,
	\end{equation*}
	there exists a sequence of integers $0 \leq n_1 < n_2 < n_3 < \dots$ such that:
	\begin{equation*}
		\begin{aligned}
			&- \ \text{For any} \ k \geq 1 \ \text{and any} \ n > n_k, \ \text{we have} \ \frac{1}{n-n_k}\sum_{i=n_k}^{n-1} a_i \leq \alpha_3.\\
			&- \ \limsup_{k \rightarrow \infty} \frac{k}{n_k} \geq \frac{\alpha_3 - \alpha_2}{\alpha_3 - \alpha_1}.
		\end{aligned}
	\end{equation*}
	\label{lem:PlissLemma}
\end{lemma}

The idea is to use Pliss' lemma in combination with Birkhoff's theorem to derive a lower bound on the measure of the set $\CP$. By obtaining good bounds on the derivatives and choosing the constants appropriately, we ensure that this set has strictly positive measure.\\

For any $N\in \bbN$, define $\beta_N := \sup_y \norm{d_yf^N}$ and $\alpha_N := \sup_y \norm{d_yf^{-N}}^{-1}$. Let $\mu$ be a measure satisfying the hypotheses of Proposition \ref{prop:MainProp}. Denoting by $\lambda^{u/s}(\mu,f^N)$ the stable and unstable Lyapunov exponents of $f^N$. We observe that:
\begin{equation*}
	\min(\lambda^u(\mu,f^N),-\lambda^s(\mu,f^N)) = N\min(\lambda^u(\mu),-\lambda^s(\mu)) > \frac{19}{20}NR(f).
\end{equation*}
On the other hand, we have:
\begin{equation*}
	\frac{NR(f)}{\max\big(\log\norm{df^N},\log\norm{df^{-N}}\big)} \rightarrow 1.
\end{equation*}
Thus, by considering a sufficiently large iterate $f^N$, condition \eqref{equivrelation} yields:
\begin{equation*}
	\min\big(-\lambda^s(\mu,f^N),\lambda^u(\mu,f^N)\big) >  \frac{19}{20}\max(-\log \alpha_N,\log \beta_N).
\end{equation*}
For simplicity, we continue to write $f, \alpha, \beta$ instead of $f^N, \alpha_N, \beta_N$. Let $\frac{19}{20} < t < 1$ such that:
\begin{equation*}
	\min\big(-\lambda^s(\mu),\lambda^u(\mu)\big) \geq t \max(-\log \alpha,\log \beta).
\end{equation*}
Define $\delta_t = t\max(-\log\alpha,\log\beta)$. Consider some constants $\sigma, \tilde{\sigma}_1, \tilde{\sigma}_2, \rho, \tilde{\rho}_1, \tilde{\rho}_2, \eta \in (0,1)$ such that $\frac{\tilde{\sigma}_i\tilde{\rho}_i}{\sigma\rho}>\sigma$. We introduce $\Delta_i$, the set of points satisfying condition $(CP_i)$ for $i \in \{1,\dots,5\}$ in Definition \ref{def:CP-NUH-Set} for these fixed constants. Note that $\Delta_i$ depends on the chosen constants, but for clarity, we omit this dependence in the notation. Recall that $\CP = \cap_i \Delta_i$. 

Since $t > 19/20$, we obtain:
\begin{equation*}
	\frac{t-4/5}{t/5} > \frac{3}{4t}.
\end{equation*}
Choose $s \in \big(\frac{3}{4t}, \frac{t-4/5}{t/5}\big)$. Define $\sigma_{t,s} = e^{-s\delta_t}$ and $\rho_{t,s} = e^{-2s\delta_t}$.
We first apply Pliss lemma to show that, for an appropriate choice of constants, $\mu(\Delta_i) > \frac{4}{5}$ for $i = 1,3$. We begin with $\Delta_1$. Recall that the ergodicity of $\mu$ and Birkhoff theorem gives us:
\begin{equation*}
	\mu(\Delta_1) = \frac{1}{n} \ \lim_{n \rightarrow \infty} \sum_{k = 0}^{n-1} \mathds{1}_{\Delta_1}(f^k(x)) \quad \text{for} \ \mu \text{-a.e.} \ x \in M.
\end{equation*}
Take such an $x$. Recall that we may assume its tangent space splits into two $df$-invariant directions: $T_xM = \Es_x \oplus \Eu_x$. Since $\Es_{x}$ is one-dimensional, we have for all $m \geq 0$:
\begin{equation*}
	\begin{aligned}
		f^m(x) \in \Delta_1 
		&\iff  \forall \ n \geq 0, \quad (\tilde{\sigma}_1)^n\leq\norm{d_{f^m(x)}f^n_{|\Es_{f^m(x)}}} \leq \sigma^n\\
		&\iff \forall \ n \geq 0, \quad \log \tilde{\sigma}_1\leq\frac{1}{n} \sum_{i = m}^{n+m-1} \log \norm{d_{f^i(x)}f_{|\Es_{f^i(x)}}} \leq \log \sigma.
	\end{aligned}
\end{equation*}
Define, for all $n \geq 0$, $a_n = \log \norm{d_{f^n(x)}f_{|\Es_{f^n(x)}}}$. Note that $a_n \geq \log \alpha$ for all $n$. By Oseledets theorem, we also have:
\begin{equation*}
	\lim_{n \rightarrow \infty} \frac{1}{n} \sum_{k = 0}^{n-1} a_k = \lim_{n \rightarrow \infty} \frac{1}{n} \log \norm{Df^n_{|\Es_x}} = \lambda^s(\mu) \leq - \delta_t.
\end{equation*}

Since $\log \sigma_{t,s} > -\delta_t$, we can apply Pliss lemma \ref{lem:PlissLemma} with constants $\alpha_1 = \log \alpha, \ \alpha_2=-\delta_t$, and $\alpha_3 = \log \sigma_{t,s}$. We obtain a sequence $0 \leq n_1 < n_2 < \dots < n_m < \dots$ such that $f^{n_k}(x) \in \Delta_1$ for every $k \in \bbN$, provided we have chosen the constants $\sigma = \sigma_{s,t}$, $\tilde{\sigma}_1 = \alpha$, and $E = \Es_x$. We conclude that:
\begin{equation*}
	\mu(\Delta_1) = \frac{1}{n} \ 
	\lim_{n \rightarrow \infty} \sum_{k = 0}^{n-1} \mathds{1}_{\Delta_1}(f^k(x)) = \lim_{m \rightarrow \infty} \frac{m}{n_m} \geq \frac{\delta_t + \log \sigma_{t,s}}{-\log \alpha + \log \sigma_{t,s}} \geq \frac{t-st}{1-st}.
\end{equation*}
Note that the condition $s<\frac{t-4/5}{t/5}$ implies:
\begin{equation*}
	\frac{t-st}{1-st} >\frac{4}{5}.
\end{equation*}
With this choice of constants, we obtain $\mu(\Delta_1) > \frac{4}{5}$. Choosing $\sigma = \sigma_{s,t}$, $\tilde{\sigma}_2 = \beta^{-1}$, and the direction $F = \Eu_x$, and applying the same argument for $f^{-1}$, we also obtain $\mu(\Delta_3) > \frac{4}{5}$.\\

We now apply the same strategy to show that $\mu(\Delta_i) > \frac{4}{5}$ for $i=2,4$. We begin with $\Delta_2$. Let $x \in M$ satisfy:
\begin{equation*}
	\lim_{n\rightarrow +\infty}\frac{1}{n}\sum_{i=0}^{n-1} \mathds{1}_{\Delta_2}(f^i(x)) = \mu(\Delta_2).
\end{equation*}
Observe that for $m \geq 0$:
\begin{equation*}
	\begin{aligned}
		f^m(x) \in \Delta_2
		&\iff \forall \ n \geq 0, \quad (\tilde{\rho}_1)^n\leq \frac{\norm{d_{f^m(x)}f^n_{|\Es_{f^m(x)}}}^2}{|\det(d_{f^m(x)}f^n)|} \leq \rho^n \\
		&\iff \forall \ n\geq0, \quad \log \tilde{\rho}_1 \leq\frac{1}{n} \sum_{i=m}^{n+m-1} \log \frac{\norm{d_{f^i(x)}f_{|\Es_{f^i(x)}}}^2}{|\det d_{f^i(x)}f|} \leq \log \rho.
	\end{aligned}
\end{equation*}
Defining $a_n = \log \frac{\norm{d_{f^n(x)}f_{|\Es_{f^n(x)}}}^2}{|\det d_{f^n(x)}f|}$, we note that:
\begin{equation*}
	|\det d_{f^n(x)}f| \leq \norm{d_{f^n(x)}f_{|\Es_{f^n(x)}}} \norm{d_{f^n(x)}f_{|(\Es_{f^n(x)})^{\perp}}}
\end{equation*}
implies that $a_n \geq \log\alpha-\log\beta$ for all $n$.

Furthermore, by Oseledets theorem, we obtain:
\begin{equation*}
	\begin{aligned}
		\lim_{n \rightarrow \infty} \frac{1}{n} \sum_{i = 0}^{n-1} a_i &= \lim_{n \rightarrow \infty} \frac{1}{n} \ \log \frac{\norm{d_xf^n_{|\Es_x}}^2}{|\det d_xf^n|}\\
		&= \lim_{n \rightarrow \infty} \frac{1}{n} \ \big( 2\log\norm{d_xf^n_{|\Es_x}} - \log|\det d_xf^n|\big)\\
		&=2\lambda^s(\mu) - (\lambda^u(\mu) + \lambda^s(\mu))\\
		&=\lambda^s(\mu) - \lambda^u(\mu) \leq -2\delta_t.
	\end{aligned}
\end{equation*}
Observe that $\log \rho_{t,s} > -2\delta_t$. Applying Pliss' lemma \ref{lem:PlissLemma} with constants $\alpha_1 = \log\alpha-\log\beta$, $\alpha_2 = -2\delta_t$, and $\alpha_3 = \log \rho_{t,s}$, and choosing $\rho=\rho_{t,s}$ and $\tilde{\rho}_1 = \alpha\beta^{-1}$, we obtain:
\begin{equation*}
	\mu(\Delta_2) \geq \frac{2\delta_t + \log \rho_{t,s}}{\log \beta - \log \alpha + \log \rho_{t,s}}\geq\frac{t-st}{1-st} > \frac{4}{5}.
\end{equation*}
By symmetry, the same argument applies to $\Delta_4$. Choosing $\tilde{\rho}_2=\alpha\beta^{-1}$, we also obtain $\mu(\Delta_4) > \frac{4}{5}$.\\

We now verify that Condition $(CP5)$ holds on a set of large measure. We claim that if $x \in \Delta_1 \cap f^{-1}(\Delta_3)$ for some constants $\sigma, \tilde{\sigma}_1, \tilde{\sigma}_2$, then there exists a constant $0< c < 1$, depending only on $\sigma$ and $f$, such that:
\begin{equation*}
	|\cos\big(\measuredangle(\Es_x,\Eu_x)\big)| < c.
\end{equation*}
Let us now prove this claim. First, note that if $x \in \Delta_1 \cap f^{-1}(\Delta_3)$, then:
\begin{equation*}
	\norm{d_xf_{|\Es_x}} \leq \sigma \quad \text{and} \quad \norm{d_xf_{|\Eu_x}} \geq \sigma^{-1}.
\end{equation*}

Take two unit vectors $e_s \in \Es_x$ and $e_u \in \Eu_x$ and define $u = e_s - e_u$. Observe that:
\begin{equation*}
	\norm{d_xfu} = \norm{d_xfe_s - d_xfe_u} \geq \norm{d_xfe_u} - \norm{d_xfe_s} \geq \sigma^{-1} - \sigma.
\end{equation*}
This implies:
\begin{equation*}
	\norm{u} = \norm{d_{f(x)}f^{-1}d_xfu} \geq \frac{\norm{d_xfu}}{\beta} \geq \frac{\sigma^{-1} - \sigma}{\beta}.
\end{equation*}
Using the classical identity for the cosine of the angle between two vectors and noting that $\norm{u} \leq \sqrt{2}$, we obtain:
\begin{equation*}
	\begin{aligned}
		|\cos\big(\measuredangle(\Es_x,\Eu_x)\big)| &= | \langle e_s,e_u \rangle | = 1 - \frac{\norm{u}^2}{2} \\
		&\leq 1 - \frac{(\sigma^{-1}- \sigma)^2}{2\beta^2} < 1,
	\end{aligned}
\end{equation*}
which proves the claim. With the choice of constant $\sigma = \sigma_{t,s}$, we conclude that $\Delta_1\cap f^{-1}(\Delta_3) \subset \Delta_5$ with $\eta = 1 - \frac{(\sigma^{-1}- \sigma)^2}{2\beta^2}$.\\
To summarize, we have proven that if we choose:
\begin{equation*}
	\begin{aligned}
		&\sigma = e^{-s\delta_t}, \quad &\tilde{\sigma}_1 = \alpha, \quad &\tilde{\sigma}_2 = \beta^{-1}, \\
		&\rho = e^{-2s\delta_t}, \quad &\tilde{\rho}_1 = \tilde{\rho}_2 = \frac{\alpha}{\beta}, \quad &\eta = 1 - \frac{(\sigma^{-1}- \sigma)^2}{2\beta^2}, \\
	\end{aligned}
\end{equation*}
where $s \in \big(\frac{3}{4t}, \frac{t-4/5}{t/5}\big)$, $t>19/20$, and $\delta_t = t\max(-\log \alpha, \log \beta)$, then we obtain:
\begin{equation*}
	\mu(\CP) = \mu(\cap_i \Delta_i)\geq \mu\big(\Delta_1\cap\Delta_2\cap\Delta_3\cap\Delta_4\cap f^{-1}(\Delta_3)\big) > 0.
\end{equation*}

We now verify that $\frac{\tilde{\sigma}_i\tilde{\rho}_i}{\sigma\rho} > \sigma$ for $i=1,2$. It is straightforward to check that, for these choices of constants, this condition holds whenever $s>\frac{3}{4t}$. This concludes the proof of Proposition \ref{prop:MainProp}.

\section{Consequences of Proposition \ref{prop:MainProp}}\label{sec:mme}

\subsection{Finiteness of Homoclinic Classes} 
The first consequence of Proposition \ref{prop:MainProp} is Theorem \ref{thm:MainThm}, which we restate here in terms of homoclinic classes.

\begin{thm}
	Let $f$ be a $\cC^2$ diffeomorphism of $M$. The number of measured homoclinic classes $\mathcal{M}(\nu)$ containing a measure $\mu \in \mathcal{M}(\nu)$ satisfying \eqref{eq:mainProp} is finite.
	\label{thm:FinitenessHomClas}
\end{thm}

\begin{proof}
	Proposition \ref{prop:MainProp} provides an integer $N$ and constants $\sigma, \tilde{\sigma}_1, \tilde{\sigma}_2, \rho, \tilde{\rho}_1, \tilde{\rho}_2, \eta$ such that, for any $\mu$ satisfying \eqref{eq:mainProp}, we have $\mu(\CP) >0$, where $\CP$ is a CP-hyperbolic set for $f^N$. To simplify the notations, we denote this set by $\Lambda$. 
	
	By Theorem \ref{thm:CP-StableManifold}, every $x \in \Lambda$ has one-dimensional stable and unstable local manifolds that are transverse at $x$ and vary continuously with $x$. Since $\Lambda$ is compact, there exists $\delta>0$ such that if $x,y \in \Lambda$ satisfy $d(x,y)<\delta$, then:
	\begin{equation*}
		\Ws_{loc}(x) \intert \Wu_{loc}(y) \neq \emptyset \quad \text{and} \quad \Ws_{loc}(y) \intert \Wu_{loc}(x) \neq \emptyset.
	\end{equation*}
	
	Again, by compactness, we can cover $\Lambda$ with a finite number of balls $B_1, \dots, B_L$ of radius $\delta$. Since $\mu(\Lambda)>0$, there exists at least one ball $B_i$ such that $\mu(B_i)>0$. Now, if $\nu_1$ and $\nu_2$ are two $f^N$-invariant measures with $\nu_j(B_i)>0$ for $j=1,2$ and some $i$, then $\nu_1\simh\nu_2$ by our choice of $\delta$. This proves that the number of $f^N$-measured homoclinic classes containing a measure $\mu$ satisfying \eqref{eq:mainProp} is finite. Finally, if $\mu$ and $\nu$ are $f$-invariant measures that are homoclinically related for $f^N$, they are also related for $f$, completing the proof.
\end{proof}

Corollary \ref{cor:FinMME} follows directly from Theorem \ref{thm:FinitenessHomClas}, since, by Ruelle's inequality, condition \eqref{hypmainthm} ensures that any measure of maximal entropy satisfies \eqref{eq:mainProp}. Furthermore, at most one measure of maximal entropy exists in each measured homoclinic class, as proven in \cite{buzzi2022measures}.\\

We emphasize that the same proof provides a uniform version of Theorem \ref{thm:FinitenessHomClas}, with a uniform bound on the number of measured homoclinic classes containing measures satisfying \eqref{eq:mainProp}.

\begin{thm}
	Let $(M,g)$ be a compact, connected, boundaryless Riemannian surface. Let $K_1,K_2>0$. There exists a constant $C$, depending only on $K_1$, $K_2$ and $M$, with the following property. For any $\cC^2$ diffeomorphism $g$ of $M$ such that $\norm{g}_{\cC^2}, \norm{g^{-1}}_{\cC^2} \in [K_1,K_2]$, the number of measured homoclinic classes containing a measure satisfying \eqref{eq:mainProp} is at most $C$.
	
	In particular, the number of m.m.e.'s of $g$ is bounded by $C$.
\end{thm}

To prove this, note that the assumption on the $\cC^2$-norm of $g$ allows us to choose the constants for the CP-hyperbolic set constructed in the proof of Proposition \ref{prop:MainProp} in a way that depends only on $K_1$ and $K_2$. Moreover, on such a set, the size of stable and unstable local manifolds is also uniformly bounded in terms of the $\cC^2$-norm, and thus by $K_1$ and $K_2$. It seems possible to obtain an explicit bound using this approach, but we will not carry out these computations here.

\subsection{SPR property} 
In \cite{buzzi2025strong}, Buzzi, Crovisier, and Sarig introduce a new class of diffeomorphisms called SPR (strongly positively recurrent).
\begin{defn}[SPR diffeomorphism]
	Let $f$ be a $\cC^r$, $r>1$ diffeomorphism of a closed surface $M$. We say that $f$ is SPR if there exists constants $0<h<h_{top}(f)$, $\tau >0$ and a Pesin block $\Lambda$ such that for all $\mu \in \bbP^s_e(f)$, we have:
	\begin{equation*}
		h(f,\mu) > h \implies \mu(\Lambda) >\tau.
	\end{equation*}
	\label{def:SPR}
\end{defn}

In their paper, Buzzi, Crovisier, and Sarig derive several consequences of the SPR property, such as the existence and finiteness of measures of maximal entropy (m.m.e.'s) and exponential mixing for m.m.e.'s. They also prove that any $\cC^{\infty}$ diffeomorphism $f$ of a closed surface satisfying $h_{top}(f) > 0$ has the SPR property.

We emphasize that if we assume $h_{top}(f) > tR(f)$ for some universal constant $t \in (0,1)$ sufficiently close to $1$, then the proof of Proposition \ref{prop:MainProp}, combined with a combinatorial application of Pliss's lemma (see Buzzi, Crovisier, and Sarig's paper), provides a simpler proof of the SPR property for diffeomorphisms with very large entropy.

\bibliographystyle{plain}
\bibliography{biblio_Finitness_mme}
\bigskip

\begin{center}
	\emph{Mat\'eo Ghezal} (mateoghezalmath@gmail.com)\\
	Laboratoire de Math\'ematiques d'Orsay\\
	CNRS - UMR 8628\\
	Universit\'e Paris-Saclay\\
	Orsay 91405, France.
\end{center}
\end{document}